\documentclass[tbtags]{amsart}
\usepackage{amsmath}%
\usepackage{amsfonts}%
\usepackage{amssymb}%
\usepackage{graphicx}

\usepackage{color}

\usepackage{bbm}
\newtheorem{theorem}{Theorem}
\theoremstyle{plain}
\newtheorem*{acknowledgement*}{Acknowledgement}

\newtheorem{claim}{Claim}

\newtheorem{conjecture}{Conjecture}
\newtheorem{corollary}{Corollary}

\newtheorem{example}{Example}

\newtheorem{lemma}{Lemma}

\newtheorem{proposition}{Proposition}
\newtheorem{remark}{Remark}

\numberwithin{equation}{section}
\newcommand\Bin{\operatorname{Bin}}
\newcommand\ff{\bar{f}}

\newcommand\dist{\operatorname{dist}}
\newcommand\Qc{Q_\mathsf c}
\newcommand\Qi{Q_\mathsf i}
\newcommand\Qo{Q_\mathsf o}
\newcommand\pin{p} %p_{in}
\newcommand\E{\mathbb{E}}
\newcommand\eqd{\overset{\mathrm{d}}{=}}
\newcommand\Po{\mathrm{Po}}

\allowdisplaybreaks
\begin{document}
\title[]{A modified bootstrap percolation on \\ a random graph coupled with a lattice}
%and inhibition}

\author{Svante Janson}
\address[Svante Janson]{Department of Mathematics, Uppsala University, \newline
\indent PO Box 480, SE-751 06 Uppsala, Sweden}
\email[]{svante.janson@math.uu.se}
%\urladdr{http://www.authorone.uni-aone.de}
\author{Robert Kozma}
\address[Robert Kozma]{Department of Mathematics, The University of Memphis, \newline%
\indent Memphis, TN, 38152, USA}%
\email[]{rkozma@memphis.edu}
\address[RK]{Department of Computer Sciences, University of Massachusetts Amherst, \newline%
\indent Amherst, MA 01003, USA}%
\email[]{rkozma@cs.umass.edu}
%\urladdr{http://www.authortwo.uni-atwo.hu}
\author{Mikl\'{o}s Ruszink\'{o}}
\address[Mikl\'{o}s Ruszink\'{o}]{Alf\'ed R\'enyi Institute of Mathematics, \newline
\indent Hungarian Academy of Sciences, \newline
\indent 13-15 Re\'altanoda utca, Budapest,  Hungary, 1053}
\email[]{ruszinko.miklos@renyi.mta.hu}
\author{Yury Sokolov}
\address[Yury Sokolov]{Department of Mathematics, The University of Memphis,\newline  
\indent Memphis, TN, 38152, USA \newline
\indent New address: Department of Medicine, UC San Diego, \newline
\indent 9300 Campus Point Drive, La Jolla CA 92037, USA 
}
\email[]{ysokolov@ucsd.edu}

%\urladdr{http://www.authorthree.uni-athree.edu}
%\thanks{2010 \emph{Mathematics Subject Classification}. 05C80; 60K35; 60C05.}
\thanks{This work is supported in part by NSF grant DMS-13-11165.
The work of S.J. is supported in part by the Knut and Alice Wallenberg
Foundation. The work of M.R. is supported in part by National Research, Development and Innovation Office grants 104343,  116769. Partial support by DARPA Superior Artificial Intelligence Grant is appreciated.}
%\thanks{This paper is in final form and no version of it will be submitted for
%publication elsewhere.}
%\date{\today}
% \date{November 6, 2015; revised July 8, 2017}

%\subjclass{Primary 05C38, 15A15; Secondary 05A15, 15A18} %
%\keywords{Bootstrap percolation; random graph; sharp threshold; mean-field}%
%\dedicatory{Dedicated to Professor XY on the occasion of his seventieth birthday.}

\begin{abstract} In this paper a random graph model $G_{\mathbb{Z}^2_N,p_d}$ is introduced, which is a combination of fixed torus grid edges in $(\mathbb{Z}/N \mathbb{Z})^2$ and some additional random ones. The random edges are called long, and the probability of having a long edge between vertices  $u,v\in(\mathbb{Z}/N \mathbb{Z})^2$ with graph distance $d$ on the torus grid is $p_d=c/Nd$, where $c$ is some constant. We show that, {\em whp}, the diameter $D(G_{\mathbb{Z}^2_N,p_d})=\Theta (\log N)$.
Moreover, we consider a modified 
%\textcolor{red}
{non-monotonous} bootstrap percolation on $G_{\mathbb{Z}^2_N,p_d}$.
We prove the presence of phase transitions in mean-field approximation and provide fairly sharp bounds on the error of the critical parameters.
%Our model addresses interesting mathematical questions of non-monotonous bootstrap percolation, and it is motivated by recent results of brain research.
\end{abstract}

\maketitle

% \definecolor{lred}{rgb}{1,0.6,0.6}
% \definecolor{lred}{rgb}{1,0.6,0.6}

\section{introduction}
%\textcolor{red}
{Bootstrap percolation is a cellular automaton, which has been introduced
%as a model of spread of infection over a discrete structure \cite{Chalupa}. In their original paper
by Chalupa, Leath, and Reich \cite{Chalupa} as a process on the Bethe lattice where every vertex can be in active or inactive state. Initially, a vertex is active with some probability independently of the state of other vertices. The process is defined so that an active vertex stays active forever, while the state of an inactive vertex at each step is determined following an
% local
update rule based on the states of its neighbors.
It is said that the process percolates if all the vertices eventually become active. }
%In this paper we omit this assumption and call the corresponding process a modified bootstrap percolation. }

%\textcolor{red}
{Bootstrap percolation on lattices has been extensively investigated in the last decades.
It has been shown under a broad range of conditions that there is a critical initialization probability such that above this probability there is percolation, while there is no percolation below this critical probability. The corresponding effect is called phase transition, which occurs at the critical probability.
The main goal is to derive conditions for the critical probability as the function of the properties of the lattice and the update rule.
%should be for a given graph before percolation happens. }
%%%\begin{equation*}
%%%p_c(G) := \inf \left\{ p \bigm| \mathbb{P}_p \left(A \text{ percolates on } G \right) \ge \frac{1}{2} \right\},
%%%\end{equation*}
%%%\textcolor{green}{where $A \subset V(G)$ is the set of initially active vertices, and $\mathbb{P}_p$ is the product probability measure on $G$.}
}
%\textcolor{red}
{For bootstrap percolation on the two-dimensional square {\em infinite} 
  lattice with 2-neighbor update rule, i.e., a site becomes active if at
  least 2 of its neighbors are active, the first result is due to van Enter
  \cite{vanEnter} who proved that the critical probability is zero. This
  result was generalized to all dimensions by Schonmann \cite{Schonmann}. It
  was shown that for the $r$-neighbor rule in $d$ dimensions the critical
  probability is 0 if $r \le d$ and 1 otherwise.} 
%\textcolor{red}
{
%However, an interesting behaviors has been observed for the process defined on grids.
The finite volume (metastabilty) behaviour was investigated by  Aizenman and Lebowitz
\cite{Aizenman} and the threshold function for the critical probabilty $p_c([n]^d,r)$ for the finite  $d$-dimensional lattice with $r$-neighbor rule has been identified up to a constant factor by Cerf and Manzo \cite{CM} for all $d\ge r$.
For $d=r=2$  the sharp threshold
\begin{equation} \label{Hol}
p_c([n]^2,2)=\frac{{\pi}^2}{18 \log n}+o\left(\frac{1}{\log n}\right)
\end{equation}
has been proved by Holroyd \cite{Holroyd}.
Surpisingly, this result contradicted to numerical predictions, which were
apparently due to slow convergence. Finally, Balogh, Bollob\'{a}s,
Duminil-Copin, and Morris \cite{Balogh12} derived sharp threshold for 
$p_c([n]^d,r)$ for all $d\ge r$.

%\textcolor{red}
{It is of interest to analyze a modified model by relaxing
  the original condition requiring that an active vertex stays active
  forever. This leads to a broader class of modified non-monotonous
  bootstrap percolation when most of techniques used in (monotonous)
  bootstrap percolation cannot be applied. There are some results for models
  with non-monotonous bootstrap percolation. Coker and Gunderson \cite{Coker} 
  considered bootstrap percolation with a modified
  $k$-threshold rule. In their case, an inactive vertex becomes active if it
  has at least $k$ active neighbors, while an active vertex with no active
  neighbors becomes inactive. The last condition allowed to generalize
  techniques previously used for bootstrap percolation and find sharp
  thresholds for the critical probability of initial activations so that all
  vertices eventually become active.} 

%\textcolor{red}
{Recently, bootstrap percolation has been considered on the Erd\H os-R\'enyi random graph  $G_{n,p}$ in \cite{Janson12}, where a theory has been developed regarding the size $a$ of the set of initially active sites.
Results include sharp threshold for phase transition for parameters $p$ and $a$, and the time $t$ required to the termination of the bootstrap percolation process.}
%\textcolor{red}
{Turova and Vallier \cite{Turova15} considered bootstrap
  percolation on the combination of the random graph $G_{n,p}$ and  the
  $n$-cycle, where random edges are added between any pair of vertices of
  the $n$-cycle with probability $p$ independently of each other. Starting
  with $a$ active vertices, they analyzed when the percolation process
  terminates. Sharp thresholds for phase transition for parameters $p$ and
  $a$ were derived. In particular, it was shown that 
for a range of the parameters,
the process percolates
  on the combined graph but not on the random graph $G_{n,p}$ without local
  edges.}  In \cite{Lengler15} the authors considered bootstrap percolation  process on $G_{n,p}$ with vertices of two different types.

%\textcolor{red}
{There has been extensive work on studying random graphs of large order, which have relatively small diameter. For example, Bollob\'{a}s and Chung \cite{Bollobas88} showed that adding a random matching to
the  \(n\)-cycle reduces its \(\lfloor n/2 \rfloor \) diameter to \((1 + o(1)) \log_{2} n\).}
The so called  \(n\)-cycle long-range percolation graph has been considered in \cite{Benjamini2001} where the probabilities of added random edges decay polynomially with the distance between the corresponding pairs of vertices.
It was shown that the diameter of this graph is of the order of
\(\log n\), assuming that the parameters are constrained to a certain parameter region. The combination of a finite \(d\)-dimensional grid $[n]^d$ with  random edges  (decreasing in distance) added has been considered by Coppersmith, Gamarnik and Sviridenko \cite{coppersmith}. They showed that under certain conditions on the dimension and probability \(p\), the diameter is either $\Theta (\log n)$ or $n^\eta$, where the power coefficient \(\eta\) satisfies $0<\eta<1$.

%\textcolor{red}
{
%Studying the diameter of random graphs of large order has great practical significance in large-scale networks.
Watts and Strogatz \cite{Watts98} introduced the {\em "small world"} network.
The edges of a so-called ring lattice with $k$ edges per vertex are rewired with probability $p$.
%on the vertex set of the \(n\)-cycle, where each vertex is connected to all vertices at distance at most $k/2$. The edges are where the edges are {\em rewired} at random with probability $p$, starting from a circle lattice with $n$ vertices and $k$ edges per vertex.
This construction leads to the drastic reduction of the network diameter, and it allows to 'tune' the graph between regularity ($p=0$) and disorder ($p=1$). A different version of the {\em "small world"} model has been described by Newman and Watts \cite{Watts99}.
Here, an \(n\)-cycle is considered and the edges of the cycle are fixed. In contrast to the original formulation, however, in \cite{Watts99} random edges are added with some probability instead of rewiring the edges of the cycle, which significantly reduces the graph diameter, too.
% This approach results in a smaller diameter for a given graph size.
Since then there has been a high interest in the small world phenomenon in mathematical and other communities \cite{albert}.
%Since the early 2000's, there is a proliferation of small-world models in various scientific disciplines \cite{albert}. These models have been analyzed using rigorous mathematical approaches describing scaling behavior and phase transitions in inhomogeneous random graphs, see, for example
% There has been a lot of interest in studying
Scaling behavior and phase transitions in inhomogeneous random graphs have been also investigated, see, for example \cite{Bollobas07}.}

%An activation process is defined on the combined $\mathbb{Z}^d$ and $G_{n,p}$ graph, whereas
%each node is at one of two possible states, either active \((1)\) or inactive \((0)\). The activity propagation is based on the following update rule for the potential \(X_v(t)\) of node $v$ at time $t$:
%%
%\[
%X_v(t+)=
%\begin{cases}
%X_v, & \text{if $v\in A(t)$;}\\
%\max \{0,X_v(t) + \sum_{v' \in A(t+)} {w_t(v',v)}\}, & \text{if $v\notin A(t)$.}
%\end{cases}
%\]
%~\\
%Conditions are derived for phase transitions and cluster sizes for graphs merging $\mathbb{Z}^2$ and classical random graphs \cite{Turova09} by linking the observed percolation process to $rank-1$ inhomogeneous random graphs. Rigorous proofs have been provided for a simplified version of the model.

In this paper we consider a stochastic process of activation propagation over the random graph which combines lattice $\mathbb{Z}^2$ with additional random edges that depend on the distance between vertices. A similar graph has been studied, by,  e.g., Aizenman,  Kesten and Newman \cite{aizenman88}, the so-called long-range percolation graph. In that model a pair of sites of the $d$-dimensional lattice $\mathbb{Z}^{d}$ is connected (or a bond is occupied) with probability that depends on the graph distance. In the present work, we change the way probabilities are defined over the long-range percolation graph, to get a sparser graph with respect to long edges.
We consider a random graph $G$ that is built as follows. We start with the $\mathbb{Z}^2$ lattice
over a $(N+1)\times(N+1)$ grid, and we assume periodic boundary conditions.
Thus, we have a torus $\mathbb{T}^2 = (\mathbb{Z}/N\mathbb{Z})^2$, with the short notation $\mathbb{Z}^2_N$.
The set
of vertices of $G$ consists of all vertices of $\mathbb{Z}^2_N$, in total
$N^2$ vertices. All the edges from the torus grid $\mathbb{Z}^2_N$ are included in
the graph $G$. In addition, we introduce random edges as follows. For every
pair of vertices
%that are at distance $d$ apart,
we assign an edge with probability that depends on the graph distance $d$
between the two vertices, i.e., $d$ is the length of the shortest path
between the given pair of vertices in the torus grid. Accordingly, the
probability of a long edge is described as follows:
\begin{equation}
\mathbb{P} \left( (u,v)\in E(G) \right)=
p_d=
\frac{c}{N}\times
{d^{-\alpha}} %=\frac{c}{Nd^{\alpha}}
\qquad \text{when}\quad  \dist(u,v)=d,
\label{defprob}
\end{equation}
%
%
%\begin{equation}
%
%p_d=\mathbb{P} \left( (u,v)\in E(G) \text{ and } dist(u,v)=d \right)=\frac{c}{N}\times \frac{1}{d}=\frac{c}{Nd},
%
%\label{defprob}
%
%\end{equation}
%
%~\\
%
where $c$  and \(\alpha\) are positive constants, $d>1$ (no multiple edges are allowed between any pair
of vertices) and $N$ is large enough so that each $p_d<1$. We assume \(\alpha = 1\)
throughout this study. We will denote this model the $G_{\mathbb{Z}^2_N,p_d}$
graph. The edges of the torus are called {\em short edges}, while the
randomly added ones are called {\em long edges}.

The present work is organized as follows: first we describe some properties of the
introduced random graph $G_{\mathbb{Z}^2_N,p_d}$. We derive bounds on the diameter of this graph and describe its degree distribution using Poisson approximation.
The second part of this paper is devoted to the study of an activation
process using a modified non-monotonous bootstrap percolation model.
First, we consider the critical probability of the activation process and state a few conjectures, 
and then to simplify the mathematical treatment, we analyze the activation as a stochastic process in mean-field approximation \cite{bbk}. 
We derive conditions for phase transitions as a function of the model parameters, including  the proportion of long edges \(\lambda\) and the $k$-neighbor update rule parameter \(k\).
%Activity propagation in graphs with two types (excitatory and inhibitory) of nodes is the objective of future studies \cite{krs}.

We will use the following standard notation; for non-negative sequences $a_m$ and $b_m$,  $a_m= O(b_m)$ if $a_m\le cb_m$ holds for some constant $c>0$ and every $m$; $a_m= \Theta (b_m)$ if both
$a_m= O(b_m)$ and $b_m= O(a_m)$ hold; $a_m\sim b_m$ if $\lim_{m\to \infty}{a_m}/{b_m}=1$; $a_m= o(b_m)$ if $\lim_{m\to \infty}{a_m}/{b_m}=0$. A sequence of events $A_n$ occurs with high probability, {\em whp}, if the probability $\mathbb{P}(A_n)=1-o(1)$.

\section{Properties of $G_{\mathbb{Z}^2_N,p_d}$}

First notice that the expected number of long edges  $E_{\ell}\subseteq E(G_{\mathbb{Z}^2_N,p_d})$ is proportional to
$N^2$.

\begin{claim}\label{edgeexp}

$\mathbb{E}(|E_{\ell}|)\sim (2c\ln 2) N^2$,\quad i.e.,\quad $\displaystyle{\lim_{N\to\infty}\frac{\mathbb{E}(|E_{\ell}|)}{2cN^2\ln 2}=1}$.

\end{claim}

\begin{proof} Indeed, the number of vertices $|\Lambda_d|$ in $\mathbb{Z}^2_N$ which are exactly at distance $d$ from a fixed vertex is
\[
|\Lambda_d| =
\begin{cases}
4d, & \text{$ 1\le d \leq \left\lfloor N/2 \right\rfloor$} \\
4(N-d), & \text{$  \left\lfloor N/2 \right\rfloor<d\le N$}
\end{cases}
\]
for $N$ odd, and
\[
|\Lambda_d| =
\begin{cases}
4d, & \text{$1\le d < N/2 $} \\
4d - 2, & \text{$ d = N/2 $} \\
4(N-d), & \text{$ N/2 < d < N $} \\
1, & \text{$ d = N $}
\end{cases}
\]
for $N$ even. The number of pairs of vertices in $\mathbb{Z}^2_N$ having distance $d$ is $\frac{N^2 |\Lambda_d|}{2}$. Therefore, for $N$ odd
\begin{eqnarray}\label{c}
\mathbb{E}(|E_{\ell}|)&=&\sum_{d=2}^N \frac{N^2 |\Lambda_d|}{2}\frac{c}{Nd}=\sum_{d=2}^{N/2}\frac{4N^2d}{2}\frac{c}{Nd}+
\sum_{d=N/2+1}^{N}\frac{4N^2(N-d)}{2}\frac{c}{Nd}
\nonumber
\\
&=&(2c\ln 2)N^2+O(N)\sim (2c\ln 2) N^2.
\end{eqnarray}

For $N$ even a similar computation  gives the same result.
\end{proof}

% Fig. 1 Case $N$ is even and odd.

\subsection{Degree distribution}
\label{subsec_degdistr}

The degree distribution of a vertex $v\in G_{\mathbb{Z}^2_N,p_d}$ with respect to long edges can be approximated by Poisson distribution. Let $W$ be the random variable describing the degree of a particular vertex $v$ considering long edges only. Then clearly, the degree of a vertex $v\in G_{\mathbb{Z}^2_N,p_d}$ considering the short edges, too, is $W+4$.

%%%%%%\begin{multline}

%%%%%%\mathbb{P} \left( \deg (v) =0 \right)

%%%%%%= \mathbb{P} (\bigcap_{d=2}^{N}(\nexists \; \text{an edge at distance $d$ from $v$}))

%%%%%%\\

%%%%%%= \prod_{d=2}^{N} (1 - p_d)^{\Lambda_d}

%%%%%%= \prod_{d=2}^{N} (1 - \frac{\delta}{2Nd})^{\Lambda_d}

%%%%%%= \prod_{d=2}^{N/2} (1 - \frac{\delta}{2Nd})^{4d} \prod_{d=N/2+1}^{N} (1 - \frac{\delta}{2Nd})^{4(N-d)}

%%%%%%\\

%%%%%%= e^{-2\delta \sum_{d=N/2+1}^{N} \frac{1}{d}} = e^{-2\delta (\ln{N} -\ln ({\frac{N}{2} + 1}))}

%%%%%%=  e^{-2\delta \ln{2}} = 2^{-2 \delta} \approx 1.

%%%%%%\end{multline}

%that is, there are exactly $k$ edges, possibly of different length.

\begin{lemma}\label{lemmaPoisAppr}
The probability that a vertex has degree $k$ considering only the long edges is given by
\begin{equation}
\mathbb{P} \left( W = k \right)
= \sum_{k_2+ \ldots + k_N =k} \;\; {\prod_{i=2}^{N}{ \binom{|\Lambda_i|}{k_i} \left( \frac{c}{Ni} \right)^{k_i} \left( 1- \frac{c}{Ni} \right)^{|\Lambda_i| - k_i} }}.
\end{equation}
The total variation distance
\begin{equation}\label{degree} d_{TV} \left( \mathcal{L}(W), \Po (\lambda)\right)=\frac{1}{2}
 \sum_{j\ge 0} | \mathbb{P}(W=j)-\mathbb{P}(Y=j)|=O(1/N),
\end{equation}
where  the random variable $Y$ has Poisson distribution $\Po(\lambda)$, with $\lambda=4c\ln2$.
\end{lemma}

\begin{proof}
The probability of the event $A_i$ that a vertex has $k_i$ edges of length $i$ is clearly
\begin{equation}
\mathbb{P} \left(A_i \right) = \binom{|\Lambda_i|}{k_i} \left( \frac{c}{Ni} \right)^{k_i} \left( 1- \frac{c}{Ni} \right)^{|\Lambda_i| - k_i}
\end{equation}
Therefore, the probability that a vertex has degree exactly $k$ is
\begin{multline}
\mathbb{P} \left( W = k \right)
= \mathbb{P} \left(\bigcup_{k_2+ \ldots + k_N =k} \;\; \bigcap_{i=2}^{N} A_i \right)
= \sum_{k_2+ \ldots + k_N =k} \;\; {\prod_{i=2}^{N}{\mathbb{P} \left(A_i \right)  }}
\\
= \sum_{k_2+ \ldots + k_N =k} \;\; {\prod_{i=2}^{N}{ \binom{|\Lambda_i|}{k_i} \left( \frac{c}{Ni} \right)^{k_i} \left( 1- \frac{c}{Ni} \right)^{|\Lambda_i| - k_i} }}.
\end{multline}

The last expression is not very convenient to use. However, a standard
Poisson approximation can be given using Le Cam's argument \cite{lecam60},
see also e.g.\ \cite{janson92}. Pick an arbitrary vertex $v$ and let enumerate
the other $N^2-5$ vertices by $u_{i}$, $i=1,\ldots,N^2-5$, excluding the
nearest neighbors, i.e., vertices at distance 1. The long edges that connect
the vertex $v$ to other vertices of the graph are independent 0--1 random
variables with Bernoulli Be$(p_i)$ distribution. In other words, let $I_i=1$
be the event that there is an edge between vertices $v$ and $u_i$, so that
$\mathbb{P} (I_i = 1) = p_i$ and $\mathbb{P} (I_i = 0) = 1 - p_i$, where
$p_i$ may in general vary for different $i$. Consider now the degree $W =
\sum_{i=1}^{N^2-5}{I_i}$ of the vertex $v$. Let
$$\lambda_1 = \sum_{i=1}^{N^2-5} {p_i}=4c\ln 2+O(1/N),$$
where the last equality follows from Eq.~(\ref{c}). By triangle inequality,
\begin{equation}
d_{TV} \left( \mathcal{L}(W), \text{Po} (\lambda)\right)
\le d_{TV} \left( \mathcal{L}(W), \text{Po} (\lambda_1)\right)+d_{TV} \left(\text{Po} (\lambda_1) , \text{Po} (\lambda)\right)
\end{equation}
The first term, by Le Cam \cite{lecam60},
see also \cite[Theorem 2.M]{janson92}, is at most
\begin{eqnarray}
\sum_{i=1}^{N^2-5} {p_i^2} &= &\sum_{d=2}^{N} {|\Lambda_d| p_d^2} \le \sum_{d=1}^{N} {|\Lambda_d| p_d^2}
= \sum_{d=1}^{N/2}{4d}\left(\frac{c}{Nd}\right)^2
\nonumber
\\
&+&
\sum_{d=N/2+1}^{N}{4(N-d)}\left(\frac{c}{Nd}\right)^2\le  \sum_{d=1}^{N}{4d}\left(\frac{c}{Nd}\right)^2=O\left(\frac{\ln N}{N^2}\right).
%- \frac{4\delta^2}{N} + \frac{\delta^3}{N},
\end{eqnarray}
and by Theorem 1.C (i) in \cite{janson92}
\begin{equation}
d_{TV} \left(\text{Po} (\lambda_1) , \text{Po} (\lambda)\right)
=O\left(|\lambda_1-\lambda|\right)
=O\left(\frac{1}{N}\right).
\end{equation}
\end{proof}
Clearly, Lemma~\ref{lemmaPoisAppr} also implies that in Eq. (\ref{degree})
each term satisfies $|\mathbb{P}(W=j)-\mathbb{P}(Y=j)|=O(1/N)$.

\subsection{The diameter of $G_{\mathbb{Z}^2_N,p_d}$}

Next we show that the addition of long edges to the torus grid reduces significantly (from
linear to logarithmic in the number of vertices) its diameter.

\begin{theorem}
There exist constants $C_1, C_2$, which depend on $c$ only, such that for the diameter $D(G_{\mathbb{Z}^2_N,p_d})$ the following hold.
\[
\lim_{N \rightarrow \infty} \mathbb{P} \left( C_1 \log{N} \leq D(G_{\mathbb{Z}^2_N,p_d}) \leq C_2 \log{N}) \right) = 1,~ {\mbox i.e.,}~  D(G_{\mathbb{Z}^2_N,p_d})=\Theta (\log N),~whp.
\]
\end{theorem}

\begin{proof}
The lower bound is trivial. The expected degree $\mathbb{E}(d(v))$ of a
vertex $v$ by Claim~\ref{edgeexp} is a constant $k=k(c)$. Thus,
the expected number of vertices $A_m$ we can reach in at most $m\ge 0$
steps from a given vertex $v$ is less than or equal to
$1+\sum_{i=1}^m k(k-1)^{i-1}$. For $m\ge3$, this is less than $k^m$, and
thus, by
Markov's inequality,
\begin{equation}\label{dlower}
\mathbb{P} (A_m \ge N^2) \leq \frac{\mathbb{E} (A_m)}{N^2} \leq \frac{k^m}{N^2}.
\end{equation}
If we choose $m\le C_1\log N$ with $C_1$ sufficiently small, the probability in Eq.~(\ref{dlower}) tends to zero, i.e., 
%\textcolor{red}
{we cannot reach all vertices of the graph from a given vertex $v$ by a path with at most $C_1\log N$ edges}. Hence,  $C_1\log N$ bounds the diameter from below.

To prove the upper bound, partition the vertices of $G_{\mathbb{Z}^2_N,p_d}$
into consecutive  $k \times k$ blocks $B_{ij}, i,j=1,\ldots,
\frac{N}{k}$, where $k$ is a constant $k(c)$ to be chosen later. (For
simplicity, we will assume that everywhere divisibility holds during the
proof; otherwise we let some blocks be $(k+1)\times(k+1)$.)
Define the graph $G^\prime$ as follows. The vertices are the
blocks, and two blocks $B_{i,j}$ and $B_{k,\ell}$, $(1,\le i,j,k,\ell\le
N/k)$ are connected iff there is a long edge from a vertex of $B_{i,j}$ to a
vertex of $B_{k,\ell}$ in $G_{\mathbb{Z}^2_N,p_d}$. We obtain a random graph
on $N^2/k^2$ vertices where the edge probabilities can be obtained from the
ones of $G_{\mathbb{Z}^2_N,p_d}$. For an arbitrary pair of vertices $B_{i,j}$
and $B_{k,\ell}$, the probability of the event $A_{i,j;k,l}$ that they are
connected is bounded from below by the probability, that two blocks which
are most distant from each other in  $\mathbb{Z}^2_N$ are connected.
Therefore, for large $N$,
\begin{eqnarray*}\label{dupper}
\mathbb{P} (A_{i,j;k,l})&\ge & \mathbb{P} (A_{1,1;N/(2k),N/(2k)})= 1 -  \mathbb{P} (\overline{A_{1,1,\frac{N}{2k},\frac{N}{2k}}})
\ge  1-(1-p_N)^{k^4}\\
&=&1-\left(1-\frac{c}{N^2}\right)^{k^4}\ge 1-e^{-ck^4/N^2}\ge ck^4/2N^2.
\end{eqnarray*}
For the second inequality we picked the two most distant vertices from each
block, and the last one follows from $e^x\le 1+x/2$ for $x<0$ sufficiently
close to $0$.
Consequently, we can couple the random graph $G'$ with a random graph
$G''\subseteq G'$ where edges appear independently with probability
$ck^4/2N^2$, i.e., $G''$ is an
Erd\H os-R\'enyi random graph $G_{n,p}$ with $n=N^2/k^2$ and $p=ck^4/2N^2$.

By, e.g., Theorem 9.b in the seminal paper of Erd\H os and R\'enyi \cite{ER} there is a constant $c_1$ such that in the Erd\H os-R\'enyi random graph $G_{n,p}$ with $p=c_1/n$ there is a giant component on at least, say, $n/2$ vertices, {\em whp}. Choosing
$$k\ge (2c_1/c)^{1/2}$$
we get that the edge probability in $G''$ is
$$ck^4/2N^2\ge c_1k^2/N^2,$$
and thus $G''$ will contain  a giant component on at least $N^2/2k^2$
vertices, {\em whp}. The diameter of the giant component of $G_{n,p}$ with
$p=c_1/n$ is known to be of order $O(\log n)$, {\em whp}. (See, e.g. Table 1
in \cite{ChLu}.) 

First, assume that vertices $u,v\in G_{\mathbb{Z}^2_N,p_d}$ are contained in
blocks $B(u)$ and $B(v)$ which are vertices of the giant component in
$G''$. Find the shortest path, say, $B(u)=B(x_0), B(x_1), B(x_2), \ldots
,B(x_m)=B(v)$, between $B(u)$ and $B(v)$ in $G''\subseteq G'$. Let $(x_0,x_1)$,
$(x_1^\prime, x_2)$, $(x_2^\prime ,x_3)$, \ldots , $(x_{m-1}^\prime,x_m)$,
$x_i,x_i^\prime\in B(x_i)$ be the edges in $G_{\mathbb{Z}^2_N,p_d}$ inducing
this path in $G'$. 

Now go from $u$ to $x_0$ in $B(u)$ along  short  ($\mathbb{Z}^2$) edges. Jump from $x_0$ to $x_1$. Then go from $x_1$ to $x_1^\prime$ in $B(x_1)$ along  short edges. Jump from $x_1$ to $x_2^\prime$, and so on.
The total length of the path from $u$ to $v$, will be at most
$$m+2k(m+1)\le (2k+1)(m+1).$$
Indeed, we make $m$ jumps, and within each block we make at most $2k$ steps along short edges. Since $m=O(\log N)$, {\em whp}, the case when $u$ and $v$ are inside blocks that belong to the giant component in $G''$ is finished.

Next we show that, {\em whp,} every vertex $v\in G_{\mathbb{Z}^2_N,p_d}$ is close to some block $B$ of the giant component in $G''$.
Indeed, by symmetry, the set $A$ of vertices in the giant component of $G''$
can be any set of vertices of the same size, with the same
probability. Therefore, one can regard $A$ as a uniformly random subset on
at least half of the vertices in $G''$. 

For some large constant $D$, the number of vertices with distance at most $D\sqrt{\log_2 N}$ from a fixed vertex $v$ in $\mathbb{Z}^2$ is
$$\sum_{d=1}^{D\sqrt{\log_2 N}}4d\ge 4D^2 \log_2 N,$$
i.e., this neighborhood contains a vertex from at least
$$ \frac{4D^2 \log_2 N}{k^2}$$
blocks.
Since $A$ contains at least half of the vertices in $G''$, the probability that none of those blocks is in $A$ is
$$\le 2^{-\frac{4D^2 \log_2 N}{k^2}}=N^{-4D^2/k^2}.$$

Therefore, the probability that there is a vertex $v\in G_{\mathbb{Z}^2_N,p_d}$ for which there is no vertex
$u$ within distance $D\sqrt{\log_2 N}$ such that $B(u)\in A$ is
$$\le N^2\cdot N^{-4D^2/k^2}<N^{-2},$$
assuming that $D$ is large enough.

Now, consider  two arbitrary vertices $u,v\in G_{\mathbb{Z}^2_N,p_d}$. If one or neither of them is in a block from $A$, then, {\em whp,} each of them can reach a block from $A$ within  $D\sqrt{\log_2 N}$ steps in $\mathbb{Z}^2$, and then proceed as in case $B(u),B(v)\in A$. Since the number of additional steps {\em whp} is $O(\sqrt{\log N})$,
the proof is finished.
\end{proof}

\section{Activation process on the random graph $G_{\mathbb{Z}^2_N,p_d}$} \label{section3}

%\textcolor{red}
{Now we introduce a stochastic process on the graph we have just built. Each vertex is described by its {\em state}, which can be either active or inactive.
% In other words, for each vertex we attribute a $2$-dimensional vector where each coordinate is a Bernoulli random variable.
%with parameters $\omega$ and $p$, correspondingly, and they are independent.
% The {\em type} is selected at the start and it remains unchanged through the process, while
The {\em state} of the vertex changes during the process according to a rule specified next.
We define a potential function $\chi_v (t)$ for each vertex $v$ such that $\chi_v(t)=1$ if vertex $v$ is active at time $t$,
and  $\chi_v (t)=0$ if $v$ is inactive.
Let $A(t)$ denote the set of all active vertices at time $t$, thus
\(A(t) = \{v \in V(G_{\mathbb{Z}^2_N,p_d})  \bigm| \chi_v (t)=1 \}\).
}

%\textcolor{red}
{At the beginning,  let $A(0)$ be  a random subset of vertices with each vertex active with probability $p$, independently of all other vertices, and the corresponding distribution we denote by $\mathbb{P}_p$.
Each vertex may change its activity based on the states of its neighbors, according to the rule ${\mathcal R_k}$
\begin{equation}
\chi_v (t+1) = \mathbbm{1}  \left(
\sum_{u \in N(v)}{\chi_u(t)}  \geq k
\right),
\label{dynamics}
\end{equation}
where
$\mathbbm{1}$ is the indicator function and
$N(v)$  denotes the subsets of vertices in the closed neighborhood of the vertex $v$, i.e., the vertex $v$ and its neighbors. Here $k$ is a nonnegative integer that specifies the threshold required for the vertex to be in the active state at the next step.
}
%For a vertex $v$ of type \(I\), the following rule holds:
%\begin{equation}
%\chi_v (t+1) = \mathbbm{1}  \left(
%\sum_{u \in N^{E}(v)}{\chi_u(t)} + \sum_{u \in N^{I}(v)}{\chi_u(t)} \geq k
%\right) = \mathbbm{1}  \left( \sum_{u \in N(v)}{\chi_u(t)} \geq k \right),
%\label{dynamics1}
%\end{equation}
%%
%where $N(v) = N^{E} (v) \cup N^{I} (v)$ is the closed neighborhood of vertex
%$v$. Notice, that vertices of type \(E\) and \(I\) have different roles and influence each other differently.

%\textcolor{red}
{
According to Eq.~(\ref{dynamics}) we have a $k$-neighbor update rule, i.e., a vertex will be active at the next time step if it has at least $k$ active neighbors including itself.
Observe, that the set of active vertices does not necessarily grow monotonically during the activation process in the present modified bootstrap percolation model, in contrast to usual bootstrap percolation.}

The set $A$ is said to percolate with respect to rule ${\mathcal R_k}$,    if eventually all vertices in $G_{\mathbb{Z}^2_N,p_d}$
get activated and stay so. The critical probability  $p_c$ for $k\le 5$ is defined as
\begin{equation}\label{criticaldef}
p_c\left(G_{\mathbb{Z}^2_N,p_d}, {\mathcal R_k}\right)
=\inf\{p:~\mathbb{P}_p~(A(0)~ percolates )\ge 1/2\}.
\end{equation}

Notice, that for $k\ge 6$, {\em whp}, even $A(0)=  V(G_{\mathbb{Z}^2_N,p_d}) $ does not percolate. Indeed, {\em whp}, the number of 
vertices with degree determined by long edges equals to zero is $>e^{-\lambda}N^2/2$. That is, even if we initially activate all of the vertices, those with degrees determined by long edges equal to zero will deactivate in the first step and stay so forever.

\section{Percolation and  density}

The pretty straightforward analysis of the activation process in terms of percolation is given as follows.

\begin{proposition} \label{P1}
For $0\le k\le 2$ and  $\lambda\ge 0$, {\em whp}, $p_c=o(1)$; for $3\le k\le 5$ and $\lambda\ge 0$, {\em whp}, $p_c=1-o(1)$.
\end{proposition}

\begin{proof}
Cases \(k=0\) and \(k=1\) are trivial. Even if we start the activation with a single vertex, it will fully percolate. 

In the case $k=2$, it is enough to show that the statement holds with $\lambda =0$, since the vertices will get activated even easier after  adding long edges.  

First notice, that in this case only isolated active vertices can get inactive. Indeed, if an active vertex $v$ is connected to some other active one, by the activation rule it will stay so forever. 

Now, let $A^\prime (0)\subset A(0)$ be the subset of initially activated, non isolated vertices. If we start the process with $A^\prime (0)$, once a vertex get activated, it will stay so forever. Indeed, if a vertex is activated, it is added to an active component, and therefore, will not be isolated. Therefore, starting the process with $A^\prime (0)$ it will be monotone.

It is left to show, that there is a `sufficiently large' random subset $A^\prime (0)$. One can easily show this concentrating on matchings in the grid. Indeed,  activate the vertices in two rounds, each time with probability $p/2$. Thus each vertex gets active with probability at most $p$. Call vertex strongly active, if it was activated in the first round, and its left neighbour in the second round. Clearly, each vertex is strongly active independently with probability $p^2/4$. Using, e.g., theorem of Holroyd (\ref{Hol}, \cite{Holroyd}) cited in the introduction concludes the proof.

To prove the case $k=3$, first partition $V(G_{\mathbb{Z}^2_N,p_d})$ into
squares 
($C_4$-s) with respect to grid edges (ignoring leftovers if $N$ is odd).
Notice, that if no vertex of a $C_4$ is initially activated and neither of
them has long edges, then the vertices of the $C_4$ will never get
activated. The probability of this event for a given $C_4$ is $\sim
e^{-4\lambda}(1-p)^4$, assuming an initial activation probability $p$. 
If follows, e.g.\ using Chebyshev's inequality, 
that {\em whp}, at least $e^{-4\lambda}(1-p)^4N^2/5$  of the vertices will
never be active, i.e., a positive fraction. The cases $k=4$, $5$ clearly follow
from the case $k=3$. 
\end{proof}

\subsection{The case $k=3$}
The evolution of the density, i.e., 
the behaviour of the random variable
$\hat{\rho}_t=\hat\rho_t(N,\lambda,\pin) =  |A(t)|/N^2$ as a function of $t$,
is particularly interesting in the case $k=3$;
in this subsection we consider only this case.
In usual bootstrap percolation, where the process is monotone, 
for an arbitrary initial configuration $A(0)$ of active vertices a final configuration
$FC(A(0))$ is always reached. This 
is  not necessarily true in our case, where oscillations may occur for ever,
as shown by the following example.

\begin{example}\label{E1}\rm
  Let $N$ be divisible by 4. Suppose first that there are no long
  edges, so the graph is $\mathbb{Z}^2_N$, and suppose that the initial
  configuration $A(0)$
is a checkerboard pattern where a vertex $(i,j)$ is active if and only if
  $i+j$ is even. Then the set of active vertices will oscillate, with
  $A(2n)=A(0)$ and $A(2n+1)=\mathbb{Z}^2_N\setminus A(0)$ for all $n$.
In this example, $|A(t)|$ is constant, but we can modify it by
adding some long edges as follows: 

Assume that there is a long edge between
$(4i,4j)$ and $(4(i+1),4j)$  for all $i,j\in \mathbb{Z}$, but no
others. 
(All coordinates are mod $N$; recall that $N$ is divisible by $4$.) 
Then, with the same checkerboard initial
configuration, the vertices $(4i,4j)$ stay active forever, while the
others oscillate as before. Hence, $\hat\rho_t$ oscillates with
$\hat\rho_{2n}=1/2$ and $\hat\rho_{2n+1}=9/16$.
\end{example}

We believe that global oscillations as exemplified in Example \ref{E1} occur
with very small probability when $N$ is large. However, there will
\emph{whp} be local oscillations, as shown by the following example.

\begin{example}\label{E2}\rm
Consider the box $Q= [-2,6]\times[-2,6]$; we partition 
$Q=\Qc\cup \Qi\cup \Qo$ where $\Qc=[0,4]\times[0,4]$ (the \emph{core}), 
$\Qi=([-1,5]\times[-1,5])\setminus \Qc$ (the \emph{inner rim})
and $\Qo=Q\setminus \Qi$ (the \emph{outer rim}).
We say that $Q$ is \emph{special}
if 
there are long edges joining 
each of the four corners of the core, i.e., $(0,0)$, $(0,4)$, $(4,0)$,
$(4,4)$, to two vertices in the outer rim, but no other long edges with an
endpoint in $Q$.

Suppose that $Q$ is special, and that in the initial configuration $A(0)$, 
every vertex in the outer rim is active, but no vertex in the inner rim,
while a
vertex $(i,j)$ in the core is active if and only if  $i+j$ is even. 
(Cf.\ Example \ref{E1}.) Then, the vertices in the outer and inner rims of
$Q$ will stay frozen as they are for ever, and so will the four corners of
the core, while the other vertices in the core will oscillate as in Example
\ref{E1}. Hence the 
total number of active vertices in $Q$ will oscillate between 45 and 48
(note that $|\Qc|=25$ and $|\Qo|=32$).

Partition $V(G_{\mathbb{Z}^2_N,p_d})$ into $9\times9$ boxes $Q_k$
(ignoring possible leftovers). Each $Q_k$ is a translate of $Q$, and we say
that $Q_k$ is special if the long edges with an endpoint in $Q_k$  are such
that $Q_k$ is a translate of a special $Q$.
For a given $\lambda>0$, each box $Q_k$ is special with some probability
$p_{s,N}$ converging to some $p_s>0$ as $N\to\infty$. Hence the expected
number of special boxes is $\sim p_s N^2/81$, and it follows easily using
Chebyshev's inequality that \emph{whp} the number of special boxes is at
least $(p_s/100)N^2$.

If $Q_k$ is a special box, then the initial configuration for $Q$ discussed
above translates to an initial configuration of $Q_k$ such that the number
of active vertices in $Q$ oscillates.
For any fixed initial activation probability  $p$, this initial
configuration of $Q_k$ has a certain positive probability, and by
independence and the law of large numbers, \emph{whp} a positive fraction of
all special boxes will have this initial configuration, and will thus oscillate.

Consequently, \emph{whp} at least a fixed positive fraction of all vertices 
participate for ever in local oscillations.
\end{example}

However, while local oscillations as in Example \ref{E2} involve many
vertices, we believe that typically, there is no global syncronization and
that therefore the many local oscillations to a large extent cancel each
other, so that the oscillations in $\hat\rho_t$ typically are small.
To make this precise,
define the random variables
$\overline\rho(N,\lambda,p)=\limsup_{t\to\infty}\hat\rho_t$
and
$\underline\rho(N,\lambda,p)=\liminf_{t\to\infty}\hat\rho_t$
(depending on the graph $G_{\mathbb{Z}^2_N,p_d}$ and $A(0)$).

\begin{conjecture}\label{c1} 
For every fixed $\lambda \ge 0$ and  initial probability $0<\pin<1 $,
there is a non-random limiting density 
$\hat{\rho}_{\lim}(\lambda ,\pin)
%=\lim_{N\to\infty}\mathbb{E}(\hat{\rho}(\lambda,N,\pin))
$ such that both $\overline\rho(N,\lambda ,\pin)$ and 
$\underline\rho(N,\lambda ,\pin)$ converge in
probability to $\hat{\rho}_{\lim}(\lambda ,\pin)$ as $N\to\infty$, i.e.,
for every $\varepsilon >0$, 
\begin{equation*}
\mathbb{P}(|\overline{\rho}(N,\lambda,\pin)-\hat{\rho}_{\lim}(\lambda,\pin)|
  >\varepsilon )\to0,
\qquad
\mathbb{P}(|\underline{\rho}(N,\lambda,\pin)-\hat{\rho}_{\lim}(\lambda,\pin)|
  >\varepsilon )\to0.
\end{equation*}
\end{conjecture}

The argument in the proof of Proposition \ref{P1} shows that if
$\delta=e^{-4\lambda}(1-p)^4/5$, then
\emph{whp} $\hat\rho_t<1-\delta$ for all $t$, 
and the same argument shows that
\emph{whp} $\hat\rho_t>\delta$ for all $t$.
Hence, \emph{whp}
$\delta\le \overline\rho(N,\lambda,p)\le 1-\delta$ and
$\delta\le \underline\rho(N,\lambda,p)\le 1-\delta$,
and if Conjecture \ref{c1} is true, then 
$0<\hat{\rho}_{\lim}(\lambda ,\pin)<1$, for every $\lambda\ge0$ and $0<\pin<1$.

If we consider expectations instead of random variables, note first
that Fatou's lemma implies 
$\liminf_{t\to\infty}\E\hat\rho_t\ge\E\underline\rho(N,\lambda,\pin)$
and
$\limsup_{t\to\infty}\E\hat\rho_t\le\E\overline\rho(N,\lambda,\pin)$.
Moreover, if Conjecture \ref{c1} holds, then
dominated convergence implies 
$\E\overline\rho(N,\lambda,\pin)\to \hat{\rho}_{\lim}(\lambda ,\pin)$
and $\E\underline\rho(N,\lambda,\pin)\to \hat{\rho}_{\lim}(\lambda ,\pin)$
as $N\to\infty$,
and thus
$$\lim_{N\to\infty}\limsup_{t\to\infty}\E\hat\rho_t
=\hat{\rho}_{\lim}(\lambda ,\pin)$$ and similarly for $\liminf$.

In the special case $\lambda=0$, 
the graph $G_{\mathbb{Z}^2_N,p_d}=\mathbb{Z}^2_N$ is regular and each closed
neighborhood has 5 elements. It follows that for $k=3$, there is a symmetry
between active and inactive vertices in the update rule, and consequently,
if we replace any initial set $A(0)$ by its complement, each $A(t)$ is
replaced by its complement $\mathbb{Z}_N^2\setminus A(t)$. 
It follows that 
$$\rho_t(N,0,1-p)\eqd 1-\rho_t(N,0,p)$$
(where $\eqd$ means equality in distribution), and thus
$\overline\rho(N,0,1-p)\eqd 1-\underline\rho(N,0,p)$.
Consequently, if Conjecture \ref{c1} holds, then
$$\hat{\rho}_{\lim}(0 ,1-p)=1-\hat{\rho}_{\lim}(0 ,p),$$
and in particular, 
$$\hat{\rho}_{\lim}(0,0.5)=0.5.$$
We  believe that
the limiting density 
%$\hat{\rho}_{\lim}(0 ,\pin)$
is less than the initial probability %$\pin$
if $\pin<0.5$ and greater than the initial probability if $\pin>0.5$:
\begin{conjecture}\label{c2} %$\hat{\rho}_{\lim}(0,0.5)=0.5$.
For $0<\pin<0.5$,~ $\hat{\rho}_{\lim}(0,\pin)<\pin$,  and for $0.5<\pin<1$, ~ $\hat{\rho}_{\lim}(0,\pin)>\pin$. 
\end{conjecture}

We believe that, furthermore, similar `critical' initial probabilities, i.e.,
where the evolution of the density changes from decreasing to increasing,
do exist for all $\lambda$. 

\begin{conjecture}\label{c3} 
For arbitrary $\lambda \ge 0$  there is a $0<p^{(crit)}_{in}(\lambda )<1$,
such that 
$\displaystyle{\hat{\rho}_{\lim}(\lambda ,p^{(crit)}_{in}(\lambda)
  )=p^{(crit)}_{in}(\lambda)}$. 
Moreover, if $0<\pin<p^{(crit)}_{in}(\lambda )$, then
$\hat{\rho}_{\lim}(\lambda ,\pin)<\pin$, and  
if $p^{(crit)}_{in}(\lambda)<\pin<1$, then $\hat{\rho}_{\lim}(\lambda ,\pin)>\pin$. 
\end{conjecture}

Numerical results shown in Figure \ref{conjecturelambda}  support our conjectures.

\begin{figure}[h]
\centering
\includegraphics[scale=.47]{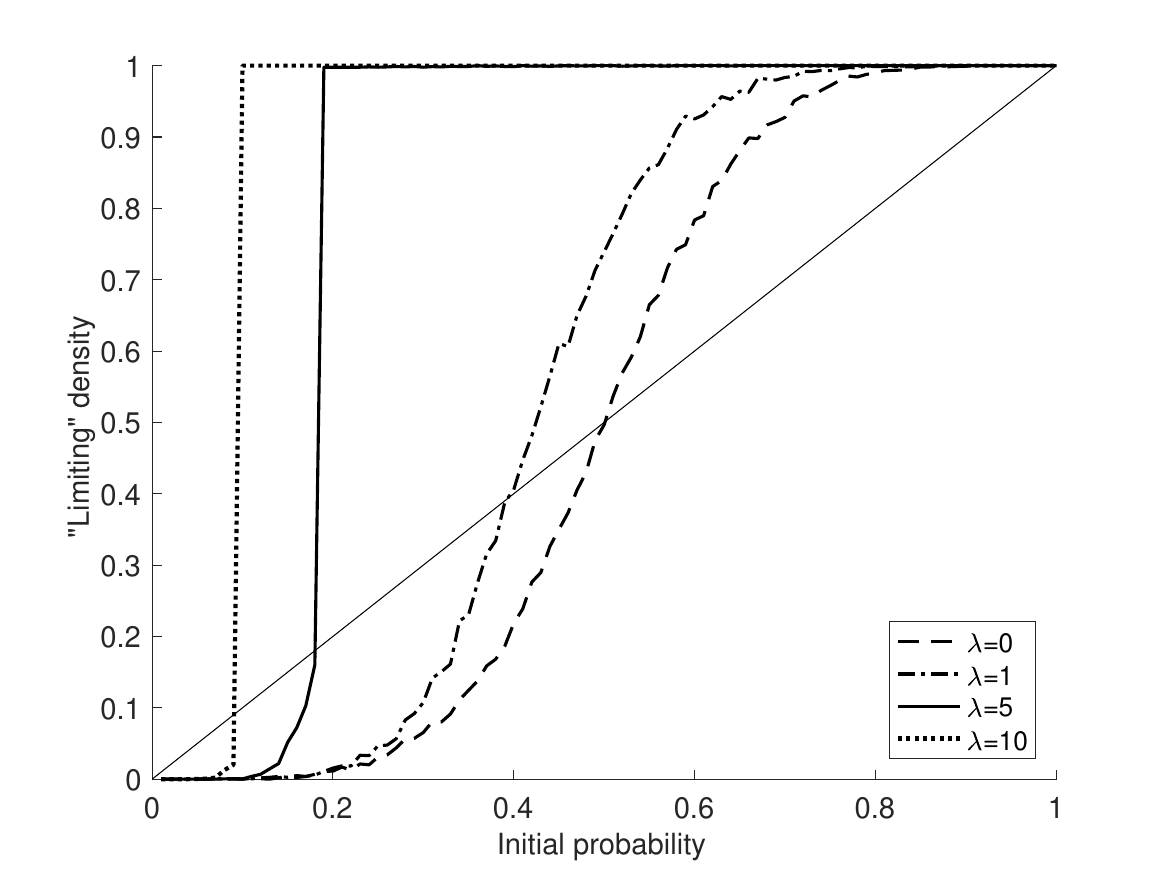}
\caption{Numerical experiments for the (conjectured) limiting density as  function of  initial probability with $N=100,$ $k=3$, and $\lambda =10,~5,~1,~0 $.
The graphs of the functions  clearly show, that for $\lambda  _1>\lambda _2$, ~  the numerical values $\hat{\rho}^{num}_{\lim}(\lambda _1,x)<\hat{\rho}^{num}_{\lim}(\lambda _2,x)$, for every $x\in [0,1]$, i.e., as $\lambda$ grows the graphs are more and more squeezed to the $y$ axis. For more details about simulations see Appendix.}
\label{conjecturelambda}
\end{figure}

%% Notice that if there are only local edges, the case $k=3$ yields {\em bootstrap percolation} with majority rule.
%
% so we are able to compare results of this subsection with the case of two types.
%The choice of small $k$ is motivated by the fact that there are vertices with degree $4$ with positive probability. Therefore, if $k$ could be $5$ or more, then there would be vertices that cannot become active unless they were activated at the beginning.}
%%%

%%%\[

%%%\phi (y) =

%%%\begin{cases}

%%%1, & \text{$ y \geq k $} \\

%%%0, & \text{$ y < k $}.

%%%\end{cases}

%%%\]

%The study of percolation process with different types of vertices is not new. AB percolation (anti-percolation) uses vertices of two types (colors) and type is defined by Bernoulli random variable, see, e.g., \cite{iyer12}. Although the definition of our model and the process is different from that of AB percolation, the last one is also of non-monotonous type. The discrete AB percolation on a connected graph is defined as follows. An edge is open if its end vertices are of different type while the edges with end vertices of the same type are deleted [cite]. Later continuum AB percolation was considered on random geometric graphs \cite{penrose}.

\section{Mean-field approximation}
% As we mentioned above, we assume that there is only one type of nodes (\(E\)).
% Moreover, for simplicity,
% In order to describe the long time behavior of the activation process defined
% in Eq.~(\ref{dynamics}) and to see the role of the distribution of the
% length of edges on the dynamics,
%\textcolor{red}
{
In order to get some more information on the evolution of the density, we consider the {\em mean-field} (MF) approximation of the activation process on $G_{\mathbb{Z}^2_N,p_d}$.
In the mean field approximation, instead of taking specific fixed neighbors of a given node,
we sample a new set of neighbors at each step \cite{bbk}.
This implies that the MF approximation does not depend on the topology of the torus, rather it is completely described by the degree distribution, and
%the cardinality of $A(t)$.
the transition probabilities from one state to another depend only on the number of active nodes.
The MF approximation means that the results derived here are obtained in the case
when the activations and degrees of the various nodes are well-mixed;
hence we ignore any dependencies between activation and vertex degrees,
as well as any dependencies between the state of a vertex and the state of its neighbors.
MF approximations are widely used in various physical models \cite{Biskup06,Talagrand11},
including systems with long-range interactions and spin glasses near critical state.
In our model, we also assume that the vertices are activated independently of each other,
ignoring the small dependencies between degrees and activities for different vertices.
}
%We call the resulting approximation of modified bootstrap percolation {\em mean-field}
% \subsection{Phase transition in mean-field model}
%Let  $\rho_t =  A(t)/N^2$, where $N^2$ is the size of the torus. Clearly, $\rho_t \in [0,1]$ and it defines the density of active nodes at time $t$.
%
%The mean-field analysis is an analytical approach of finding phase
%transitions in the stochastic process by averaging the system over
%space. Thus, the mean-field model reduces the analysis of a system with
%distributed components to a system with a single component.
%Let $f(\rho_t)$ denote the conditional mean of $\rho_{t+1}$ given
%$\rho_t$,
%for the mean-field approximation.

\subsection{Basic concepts}
%\textcolor{red}
The mean-field density $\rho_t$ is defined as follows. Start with $\rho_0=\hat{\rho}_0 \approx p_{in} $. Recall that $\deg(v)$ denotes the degree with respect to the long edges only, so the total degree of a vertex $v$ is $\deg(v)+4$. $\rho_t$ is given by the following stochastic recursion
\begin{equation}\label{rho+}
N^2 \rho_{t+1} = \Bin(N^2 \rho_{t},f^{+}(\rho_{t})) + \Bin(N^2 (1-\rho_{t}),f^{-}(\rho_{t})),
\end{equation}
where
\begin{align}
f^{+}(x) &=  \sum_{n=4}^{N^2-1} \mathbb{P} \left( \deg (v) = n-4 \right)
\sum_{i = k}^{n+1} \binom{n}{i-1} x^{i-1} (1-x)^{n-i+1},
\label{f+}
\\
f^{-}(x) &= \sum_{n=4}^{N^2-1} \mathbb{P} \left( \deg (v) = n-4 \right)
\sum_{i =k}^{n} \binom{n}{i} x^{i} (1-x)^{n-i}.
\label{f-}
\end{align}

\begin{lemma}\label{L2}
%\textcolor{red}
{Under the mean-field assumptions for the defined process on $G_{\mathbb{Z}^2_N,p_d}$, $\rho_t$ is a Markov process describing the probability that a given vertex $v$ is active at $t$, i.e., $\rho_t$ approximates the density $\hat{\rho}_t$.
Moreover, given $\rho_t$, $\rho_{t+1}$ has mean $f(\rho_t)$ and variance
$g(\rho_t)/N^2$ where}
\begin{align}
f(x) &= xf^{+}(x) + (1-x)f^{-}(x)  , \label{xf}
\\
g(x) &= xf^{+}(x)(1-f^{+}(x)) + (1-x)f^{-}(x) (1-f^{-}(x)).
\end{align}
\end{lemma}

\begin{proof}
%\textcolor{red}
{At the beginning, for a given initialization probability
  $p_{in}$, $\rho_0=\hat{\rho}_0 \approx p_{in}$ since vertices are initialized
  independently at random. Under MF assumptions the state of each vertex at
  time $t$ is a Bernoulli random variable with parameter $\rho_t$;
  furthermore, different vertices are regarded as independent. The rest of
  the lemma follow immediately from \eqref{rho+}--\eqref{f-}.} 
\end{proof}

\begin{remark}
  In our model, the activation of a vertex is deterministic given the number of
  active vertices in the closed neighborhood. More generally,
one can consider a model where an active (inactive) vertex with $i$ active
neighbors is activated with some probability $p_i^+$ ($p_i^-$), where
$p_i^\pm$ are some given probabilities. In this more general case,
\eqref{f+}--\eqref{f-} become
\begin{align}
f^{+}(x) &=  \sum_{n=4}^{N^2-1} \mathbb{P} \left( \deg (v) = n-4 \right)
\sum_{i = 1}^{n+1} p_i^+\binom{n}{i-1} x^{i-1} (1-x)^{n-i+1},
%\label{f++}
\\
f^{-}(x) &= \sum_{n=4}^{N^2-1} \mathbb{P} \left( \deg (v) = n-4 \right)
\sum_{i =0}^{n}p_i^- \binom{n}{i} x^{i} (1-x)^{n-i}.
%\eqref{f--}
\end{align}
\end{remark}

Lemma~\ref{L2} shows that the conditional variance of $\rho_{t+1}$ is
$g(\rho_t)/N^2=O(N^{-2})$, since $g\in [0,1]$ for any $\rho_{t} \in [0,1]$; thus $\rho_{t+1}$ is well concentrated
for large $N$, and we can approximate $\rho_{t+1}$ by the mean $f(\rho_t)$.

The function $f(\cdot)$ given by \eqref{xf} can be simplified to
\begin{equation}
  \begin{split}
f(x)&=xf^{+}(x) + (1-x)f^{-}(x)
\\&
=
\sum_{n=4}^{N^2-1} \mathbb{P} \left( \deg (v) = n-4 \right) \sum_{i =k}^{n+1} \binom{n}{i-1} x^{i} (1-x)^{n-i+1}
\\
&\qquad+
\sum_{n=4}^{N^2-1} \mathbb{P} \left( \deg (v) = n-4 \right) \sum_{i= k}^{n}  \binom{n}{i} x^{i} (1-x)^{n-i+1}
\\&
=
\sum_{n=4}^{N^2-1} \mathbb{P} \left( \deg (v) = n-4 \right) \left(\sum_{i= k}^{n+1}  \binom{n+1}{i} x^{i} (1-x)^{n-i+1} \right)	.
  \end{split}
\end{equation}
This can also be seen directly. Namely, if $v$ has \(n-4\) long edges, the closed neighborhood of $v$ contains $n+1$ vertices, of which $k$ have to be active for activation of $v$, and in the MF approximation, these $n+1$ vertices are active independently of each other.

In Section~\ref{subsec_degdistr} we showed that the degree distribution can be approximated by Poisson distribution \(\Po(\lambda)\).  We use this fact to approximate $f(x)$.
Consider the function
\begin{equation}
\ff(x) =
\ff_k(x)
=\sum_{n=4}^{\infty} \frac{e^{-\lambda} \lambda^{n-4}}{(n-4)!}  \sum_{i= k}^{n+1}  \binom{n+1}{i} x^{i} (1-x)^{n-i+1}.
\label{functionf}
\end{equation}
The difference between $f(x)$ and $\ff(x)$ can be bounded by
\begin{multline}\label{diffff}
| f(x) - \ff(x) |
\le \sum_{n=4}^{\infty} \left| \mathbb{P} \left( \deg (v) = n-4 \right) - \frac{e^{-\lambda} \lambda^{n-4}}{(n-4)!} \right|  \sum_{i= k}^{n+1}  \binom{n+1}{i} x^{i} (1-x)^{n-i+1}
\\
\le \sum_{n=4}^{\infty} \left| \mathbb{P} \left( \deg (v) = n-4 \right) - \frac{e^{-\lambda} \lambda^{n-4}}{(n-4)!} \right| = O\left(\frac{1}{N}\right)
\end{multline}
where the last equality follows from Lemma$~1$.
%Note also that $ f(0) = \bf(0)=0 $.

\subsection{Derivation of criticality for various \(k\) values}
We assume for simplicity that $k$ is at most $5$ in the present study. 

%The choice of \(k\leq3\) is motivated by the fact that \(k=3\) corresponds to the classical majority update rule in the two-dimensional square grid. Majority update rules are widely studied in the literature.
% and in this work we focus our attention to such cases.

We rewrite $\ff=\ff_k$ defined in \eqref{functionf} as
\begin{equation}
\label{f_ordo}
\ff_k(x)
= \sum_{n=0}^{\infty} \frac{e^{-\lambda} \lambda^{n}}{n!}
\left(\sum_{i= k}^{n+5}  \binom{n+5}{i} x^{i} (1-x)^{n+5-i} \right) = \mathbb{P} [\Bin( \deg(v) + 5,x) \ge k],
\end{equation}
where random variable $\deg(v) \sim \Po(\lambda)$.

%For $k=0,1,2$ and $3$, respectively,
%the internal sums in Eq.~\eqref{f_ordo} are expressed as
%\begin{align}
%\sum_{i= 0}^{n+5}  \binom{n+5}{i} x^{i} (1-x)^{n+5-i}&= 1,
%\\
%\sum_{i= 1}^{n+5}  \binom{n+5}{i} x^{i} (1-x)^{n+5-i}&= 1 - (1-x)^{n+5},
%\\
%\sum_{i= 2}^{n+5}  \binom{n+5}{i} x^{i} (1-x)^{n+5-i}
%&=1 - (1-x)^{n+5} - (n+5)(1-x)^{n+4} x,
%\\
%\sum_{i= 3}^{n+5}  \binom{n+5}{i} x^{i} (1-x)^{n+5-i}
%&=
%1 - (1-x)^{n+5} - (n+5)(1-x)^{n+4} x
%\\&\qquad- \frac{(n+5)(n+4)}{2}(1-x)^{n+3} x^{2}. \notag
%\end{align}
%Hence, \eqref{f_ordo} yields, by simple calculations,
%\begin{align}
%\ff_{0}(x)&=\sum_{n=0}^{\infty} \frac{e^{-\lambda} \lambda^{n}}{n!}  = 1 ,
%\label{func0}
%\\
%\ff_{1}(x)&=
%\sum_{n=0}^{\infty} \frac{e^{-\lambda} \lambda^{n}}{n!} \left(  1 - (1-x)^{n+5} \right)  = 1 - e^{-\lambda x} (1-x)^{5},
%\label{func1}
%\\
%\ff_{2}(x)&
%=\sum_{n=0}^{\infty} \frac{e^{-\lambda} \lambda^{n}}{n!} \left(  1 - (1-x)^{n+5} - (n+5)(1-x)^{n+4} x \right)
%\\& \notag
%= 1 - e^{-\lambda x} \left( (1-x)^{5} + 5x(1-x)^{4} + \lambda x (1-x)^{5} \right) ,
%\label{func2}
%\\
%\ff_{3}(x)&=\sum_{n=0}^{\infty} \frac{e^{-\lambda} \lambda^{n}}{n!} \biggl( 1 - (1-x)^{n+5} - (n+5)(1-x)^{n+4} x
%\\&\hskip8em \notag
%- \frac{(n+5)(n+4)}{2}(1-x)^{n+3} x^{2} \biggr)
%\\& \notag
%= 1 - e^{-\lambda x} \biggl( (1-x)^{5} + 5x(1-x)^{4} + \lambda x (1-x)^{5} +  \frac{\lambda^{2}}{2} x^{2} (1-x)^{5}
%\\&\hskip8em \notag
%+ 5\lambda x^{2} (1-x)^{4} + 10x^{2}(1-x)^{3} \biggr).
%\label{func3}
%\end{align}

%The fixed points for $\rho_t$ are the solutions to $f(x) = x$.

%\textcolor{red}
{
The critical probabilities in the mean-field approximation are given by the solutions to the fixed point equation $x=f(x)$, where the solutions of this equation are called fixed points.
%The main task of the mean-field approximation is to find solutions to the fixed-point equation $x=f(x)$, where the solutions of this equation are called fixed points.
This approach is based on the observation that the critical behavior of the original system often occurs near the unstable fixed points of mean-field approximation \cite{Biskup06, Talagrand11}. For a discrete time dynamical system, a fixed point is called stable if it attracts all the trajectories that start from some neighborhood of the fixed point. Otherwise, a fixed point is unstable. If $f(x)$ is continuously differentiable in an open neighborhood of a fixed point $x_0$, a sufficient condition for $x_0$ to be stable or unstable is $|f'(x_0)|<1$ or $|f'(x_0)|>1$, respectively; see, e.g., \cite{hirsch12}.}

\begin{proposition}\label{un} Let $\ff_{k}(x) : [0,1] \rightarrow [0,1]$ be
  the family of maps for $k=0, \ldots, 5$ defined by
  \eqref{f_ordo}. These maps have the following fixed points
for any $\lambda>0$:
\begin{itemize}
\item [(i)] for $k=0$ the only fixed point is $1$ and it is stable. %there is a stable fixed point $1$;
\item [(ii)] for $k=1$ there are two fixed points: $1$ is stable and $0$ is unstable. %an unstable fixed point $0$ and a stable fixed point $1$;
\item [(iii)] for $k=2,3,4$ there are three fixed points: $0$ and $1$ are stable \\
and $x_k(\lambda )\in (0,1)$ is unstable;
\item [(iv)] a. for $k=5$ there are three fixed points for $\lambda > \ln(5)$: $0$ and $1$ are stable and $x_5(\lambda )\in (0,1)$ is unstable; \\
b. and there are two fixed points for $\lambda \leq \ln(5)$: $0$ is stable and $1$ is unstable.
\end{itemize}
\end{proposition}
%%
% The fixpoint equations  $\ff_{k}(x) = x$ can be written for each $k=0,1,2,3$ as follows.

%\subsection*{$\mathbf{k=0}$}

\begin{proof}
For $k=0$, the equation $\ff_0(x) = x$ reduces to just
\begin{equation}
x = 1.
\label{mfeqn0}
\end{equation}
In this case the fixed point $x=1$ is stable since $\ff'_0(x) =0$.

For $k=1$, $\bar{f}_1(x) = x$ can be written
\begin{equation}
(1-x)e^{\lambda x} = (1-x)^{5}.
\label{mfeqn1}
\end{equation}
This equation has only two solutions $0$ and $1$ in $[0,1]$, where $0$ is an unstable fixed point since $\bar{f}'_1(0) =5+ \lambda > 1$, while $1$ is a stable fixed point because $\bar{f}'_1(1)=0$.

Now we consider cases (iii) and (iv)-$a$ together. It is easy to see that in these cases $\ff_{k}(0) = 0$ and $\ff_{k}(1) = 1$. 
% Based on the definition of function $\ff_{k}(x)$ in \eqref{f_ordo}, one can define a distribution function $F_k(x)$ as follows:  $F_k(x) = \ff_{k}(x)$ for $x\in[0,1]$, $F_k(x)=0$ for $x<0$ and $F_k(x) =1$ for $x>1$. Moreover, $F'_k(x)$ is a probability density function. 
Also easy calculations show that $\ff'_{k}(x)$ is given on $(0,1]$ by
\begin{eqnarray}
\label{FuncPrime}
\hspace{.5in} 
\ff'_{k}(x) = \frac{k}{x} \mathbb{P} [ \Bin( \deg(v) + 5,x) = k ] = \frac{k}{x} \mathbb{P} [ \Po(\lambda x) + \Bin(5,x) = k].
\end{eqnarray}
This function can be rewritten (for any $k \ge 2$) as
\begin{eqnarray}
\label{FuncPrimeSimpl}
\hspace{.5in} 
\ff'_{k}(x) = k e^{-\lambda x} x^{k-1} \sum_{i=0}^{\min\{k,5\}} \binom{5}{i}  \frac{\lambda^{k-i} (1-x)^{5-i}}{(k-i)!}.
\end{eqnarray}
In order to see that there exists a solution of $\ff_k(x) = x$ on $(0,1)$, note that in case (iii) $\ff'_{k}(0) = 0$ and $\ff'_{k}(1) = 0$. In case $k=5$, $\ff'_{5}(1) = 5 e^{- \lambda}$, which is less than 1 if $\lambda > \ln(5)$, while $\ff'_{5}(0) = 0$  for any $\lambda$. Since function $\ff_{k}(x)$ is continuous there will be at least one solution to $\ff_k(x) = x$ on $(0,1)$.

This solution is unique. Assume for the contrary that there exist at least
two solutions on $(0,1)$.
Since $0$ and $1$ are solutions, Rolle's theorem implies that
the derivative $\ff'_k(x)-1$
of $\ff_{k}(x) - x$ would have at least three zeros on $(0,1)$. 
We are going to show that the function $\ff'_{k}(x)$ is unimodal on $[0,1]$,
and not constant on any interval, which would yield a contradiction.
% Recall that a log-concave probability density is unimodal; hence, it is sufficient to check that $F'_k(x)$ is log-concave, i.e., the logarithm of the function is concave. 
To establish the required property of $\ff'_{k}(x)$, we denote the quintic polynomial in \eqref{FuncPrimeSimpl} by $P_k(x)$. 
Clearly, $e^{-\lambda x}$,  $ x^{k-1} $ are log-concave on $(0,1)$, and $P_k(x)$ is strictly log-concave on $(0,1)$, see Appendix. 
Hence $\ff'_{k}(x)$ is strictly log-concave, and therefore it is unimodal,
and not constant in any interval.
%as the product of log-concave functions is log-concave. 
%This implies that the distribution is unimodal. 
%Since the set of modes is always a closed interval, in our case it is a single point. 
 
 In the existence argument above we showed that $\bar{f}'_k(0)=0$ and $\bar{f}'_k(1) < 1$. Therefore, the fixed points $x=0$ and $x=1$ are stable in cases (iii) and (iv)-$a$. This also implies that the unique solution on $(0,1)$ is unstable.

Finally, in case (iv)-$b$, $\ff_{5}(0) = 0$ and $\ff_{5}(1) = 1$ for all
$\lambda \ge 0$. Under the condition on $\lambda$, we still have that
$\ff'_{k}(0) = 0$. However,  $\ff'_{k}(1)  \ge 1$ for $\lambda \le
\ln(5)$. Hence, if we had at least one solution of $\ff_{5}(x)=x$ on
$(0,1)$, then $\ff'_{5}(x)-1 = 0$ would have three solutions in $(0,1]$,
contradicting to the fact that $\ff_{5}(x)$ is unimodal 
and not constant on any interval.
The fixed
point $x=0$ is stable since $\ff'_{k}(0) = 0$, and $x=1$ is unstable because
$\ff'_{k}(1) >1$ for $\lambda < \ln(5)$. When $\lambda=\ln(5)$ we have
$\ff'_{5}(1) = 1$ which does not imply immediately the stability type of the
fixed point. However, since there is no solutions to $\ff_{5}(x) = x$ on
$(0,1)$ and $x=0$ is a stable fixed point, for $\lambda=\ln(5)$ the fixed
point $x=1$ is unstable. 
\end{proof}

For all cases considered above $0$ is a fixed point of $\ff$.
As we noted before, the error $f(x)-\ff(x)$ is $0$ at $0$, so this fixed
point is also a fixed point of $f(x)$ for any $N$.
If $x$ is an unstable fixed point of $\ff$ with $\ff'(x) > 1$, then
\eqref{diffff} implies that $f(x)$ has a fixed point shifted from $x$  at
most by $O(1/N)$.
%At the fixed point $1$ in the case of finite $N$ with the presence of the error term one may expect diffrent scenarios, the fixed point is shifted but it is still in $[0,1]$ or it is out of the domain of $f(x)$. In the first case, one gets almost spread of activation, that is, $\rho_t = 1 - O(1/N)$ and so $A(t) = N^2 - O(N)$ ?! In the second case, the process will spread out since the process is monotone but will fluctuate near density to be $1$.
These arguments are valid in case $\lambda$ is a fixed constant independent of $N$.

Let $p$ denote the probability that a node is initially activated and $p_{c}$
be the nontrivial solution(s) derived above. Since $\rho_t$ is a Markov
process, for the mean-field approximation we obtain the following theorem.
%where 'with high probability' means with probability $1-o(1)$ as $N\to\infty$.

\begin{theorem}\label{ThmMF1}
In the mean-field approximation of the activation process \(A(t)\) over
random graph $G_{\mathbb{Z}^2_N,p_d}$
there exists a critical probability \(p_c\) such that
for a fixed $p$, with high probability for large $N$,
all vertices will eventually be active if $p> p_c$, while all vertices will
eventually be inactive for $p < p_c$.
The value of \(p_c\) is given as the function of \(k\) and $\lambda$
as follows:
\begin{itemize}
\item[(i)]
For $k=0$ and any $\lambda$, $p_c=0$ and all vertices will become active in
one step for any $p$.
%\item[(ii)]
%For $k=2$ and $\lambda \ll N$, $ p_c \in (x_{2}(\lambda) - O(1/N) ,
%x_{2}(\lambda) + O(1/N))$, where $x_{2}(\lambda)$ is a nontrivial solution
%to $x=\bar{f}_2(x)$.
\item[(ii)]
For $k=1$ and any $\lambda$, $p_c = 0$, i.e., for any fixed $p > 0$,  all vertices will eventually become active with high probability.
\item[(iii)]
For $k=2,3,4$ and any $\lambda$, $ p_c = x_{k}(\lambda)$, where
$x_{k}(\lambda)\in(0,1)$ is a nontrivial solution to $x=\bar{f}_k(x)$.
\item[(iv)]
For $k=5$ and $\lambda >\ln(5)$, $ p_c = x_{5}(\lambda)$, where
$x_{5}(\lambda)\in(0,1)$ is a nontrivial solution to $x=\bar{f}_5(x)$; for $\lambda \le \ln(5)$, $ p_c =1$.
\end{itemize}
%In the case $N \rightarrow \infty$ $p_c = x_{nt}$ and all vertices will be active if $p> p_c$.
\end{theorem}

\begin{proof}
  Consider the case $0\le p<p_c$ (and thus (iii) or (iv)); the case $p_c<p\le 1$ is similar and (i) and (ii) are trivial.
In the limit as $N\to\infty$, $\rho_0=p$ and $\rho_t$ is
deterministic with $\rho_{t+1}=\ff(\rho_t)$. Since $p<p_c$, the sequence
$\rho_t=\ff^t(p)$ converges, as $t\to\infty$, to the fixpoint $0$. Furthermore,
because $\ff'(0)=0$, the convergence is (at least) quadratic, and in
particular geometric.

Now consider a fixed positive integer $N$. The deterministic sequence $\ff^t(p)$ just considered reaches below $1/N$ for $t\ge t_N$, where $t_N=O(\log N)$. The sequence $\rho_t$ is a random perturbation of $\ff^t(p)$. In each step, we have two sources of error: the difference in mean
$f(\rho_t)-\ff(\rho_t)=O(1/N)$, by \eqref{diffff}, and the random error
coming from the binomial distributions in \eqref{rho+}, which by a standard
Chernoff bound is $O(N^{-0.9})$ with probability $1-O(N^{-1})$, say.
Since further $|f'(x)|\le1$ for small $x$, the combined error from the first
$t_N$ steps is $t_N(O(N^{-1})+O(N^{-0.9}))=O(N^{-0.8})$ with probability
$1-O(t_N N^{-1})=1-o(1)$.
Hence, with high probability, we reach a state with $\rho_t=O(N^{-0.8})$.
Then $f(\rho_t)=O(\rho_t^2)=O(N^{-1.6})$, and by another Chernoff bound (or
Chebyshev's inequality), $\rho_{t+1}=O(N^{-1.6})$ with high probability.
But then $f(\rho_{t+1})=O(\rho_{t+1}^2)=O(N^{-3.2})$, and thus (conditionally
given $\rho_{t+1}$), the expected number of active vertices at time $t+2$
is $N^2f(\rho_{t+1})=O(N^{-1.2})=o(1)$, and thus with high probability there
are no active vertices at all at time $t+2$.
\end{proof}

\begin{corollary}\label{bound}
Case $(iii)$ of Theorem \ref{ThmMF1} can be sharpened as follows.
\begin{itemize}
%\item[(iii)]
\item[]
For $k=2$ and any $\lambda$, $ p_c = x_{2}(\lambda)$, where
$x_{2}(\lambda)\in (0,x_2(0)]$ is a unique solution to $x=\bar{f}_2(x)$ and $x_2(0) \approx 0.131$.
\item[]
For $k=3$ and any $\lambda$, $ p_c = x_{3}(\lambda)$, where
$x_{3}(\lambda)\in (0,x_3(0)]$ is a unique solution to $x=\bar{f}_3(x)$ and $ x_3(0)= 0.5$.
\item[]
For $k=4$ and any $\lambda$, $ p_c = x_{4}(\lambda)$, where
$x_{4}(\lambda)\in (0,x_4(0)]$ is a unique solution to $x=\bar{f}_4(x)$ and $ x_4(0)=1-x_2(0) \approx 0.869$.
\end{itemize}
\end{corollary}

\begin{proof}
The values $x_2(0) = \frac{11}{12} - \frac{1}{12}(235+6
\sqrt{1473})^{1/3}-\frac{13}{12}(235+6\sqrt{1473})^{-1/3} \approx 0.131123$,
 $x_3(0)=\frac12$, and 
$ x_4(0) = \frac{1}{12} + \frac{1}{12}(235+6
\sqrt{1473})^{1/3}+\frac{13}{12}(235+6\sqrt{1473})^{-1/3} \approx 0.868877$,
 can be obtained from $x=\bar{f}_k(x)$ with $\lambda = 0$.
%(\ref{mfeqn2}) and (\ref{mfeqn3}), respectively.
Clearly,
$ p_c = x_{k}(\lambda)$ is a non-increasing function of
$\lambda\ge0$. 
Indeed, we can couple
two models with parameters $\lambda_1$ and $\lambda_2$,
with $\lambda_1<\lambda_2$,
such that the density of active vertices for $\lambda_1$ is less than or equal to the density of active vertices for $\lambda_2$. Therefore, for $\lambda \ge 0$, 
$x_k(\lambda) \le x_k(0)$.
\end{proof}

From the equation (\ref{f_ordo}),    $p_c=x_k(\lambda)\to0$ as $\lambda\to\infty$.
As $\lambda \rightarrow 0$,  $p_c=x_k(\lambda)$ tends to 1 for \(k=5\), 0.868877 for \(k=4\), 0.5 for \(k=3\), and 0.131123 for \(k=2\).

Comparing Figures \ref{pcLambda} and  \ref{pcLambdaReal}  one can see, that in cases $k=3,4,5$ the critical values obtained in MF seem to approximate well the (numerical) "limiting" densities formulated in Conjectures \ref{c1}, \ref{c2} and \ref{c3} (if they exist), i.e., the threshold where evolution of the density is changing from decreasing to increasing {\em in the real model.}

%(Here I will add another Figure!!!)

%\begin{figure}[!t]
%\center{\includegraphics[width=.75\linewidth]{k2to5.pdf}}
%\caption{$p_c$ as a function of $\lambda$ for $k=2,\ldots, 5$.}
%\label{pcLambda}
%\end{figure}

\begin{figure}[!t]
\begin{center}
\begin{minipage}[h]{\linewidth}
\center{\includegraphics[width=0.7\linewidth]{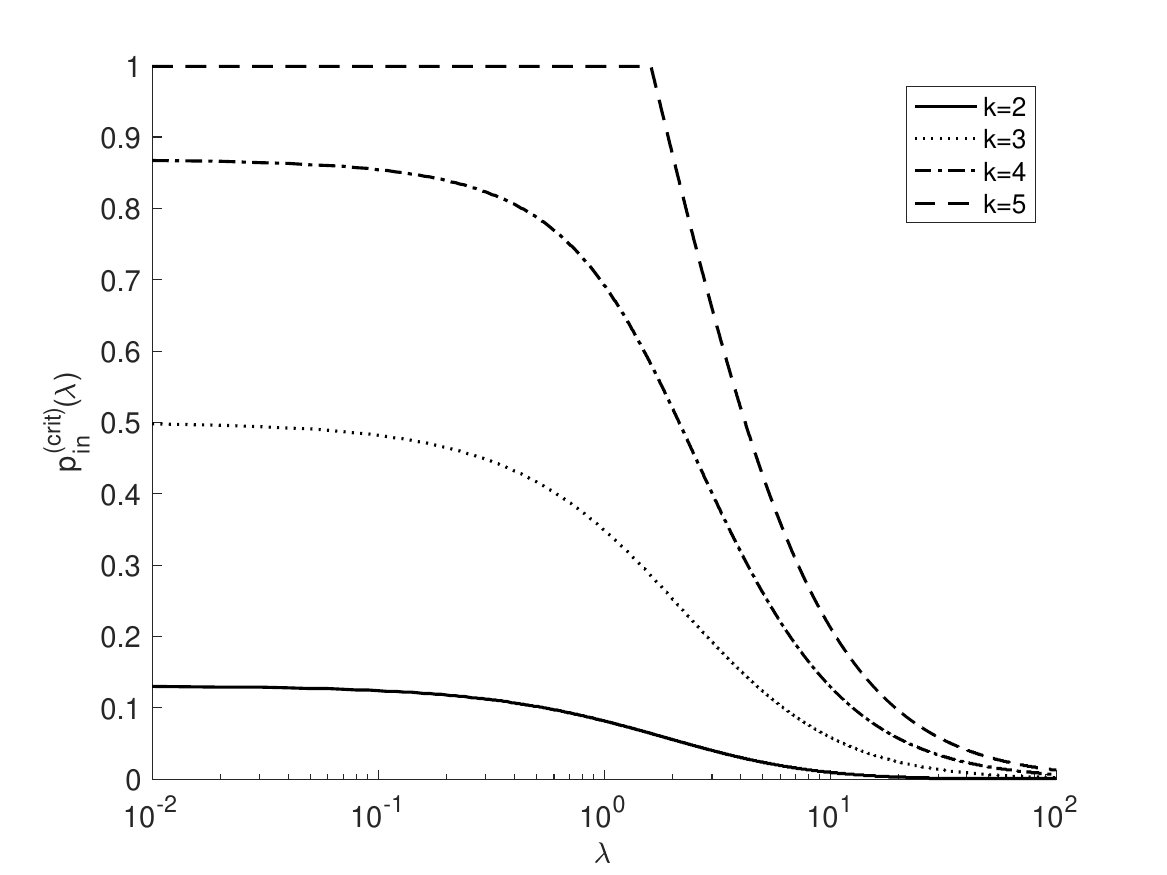}}
\caption{Mean-field approximation: $p_c$ as a function of $\lambda$ for $k=2,\ldots, 5$, that is, the numerical solution of $\ff_{k}(x) = x$.} 
\label{pcLambda}
\end{minipage}
\vfill 
\begin{minipage}[h]{\linewidth}
\center{\includegraphics[width=0.7\linewidth]{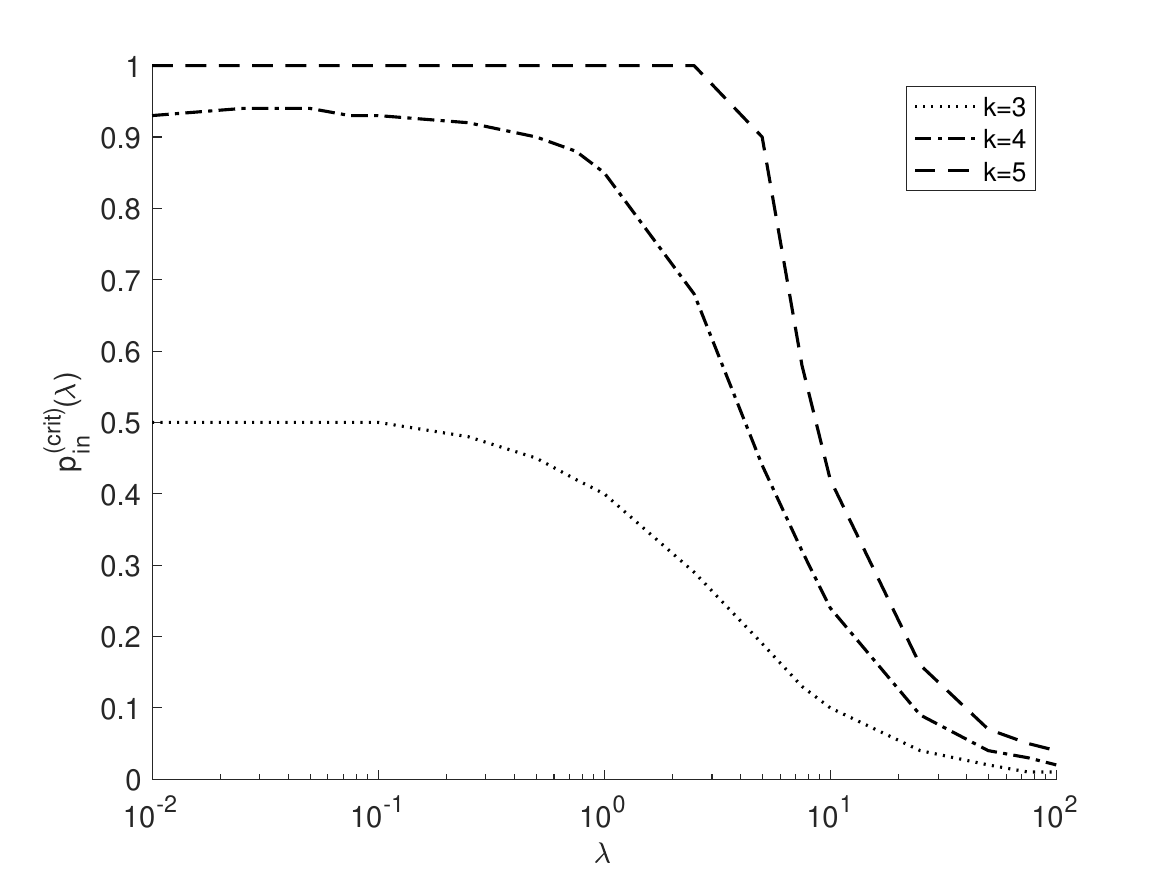}}
\caption{Real process: numerical values for $p^{(crit)}_{in}(\lambda )$ as functions of $\lambda$ for $k=3,\ldots, 5$ with $N=100$. Notice that here  $p^{(crit)}_{in}(\lambda )$ is rather the threshold where the evolution of the density changes from decreasing to increasing. For more details about simulations see Appendix.}
\label{pcLambdaReal}
\end{minipage}
\end{center}
\end{figure}

\section{Discussion and Conclusions}

In this work we introduced the random graph model $G_{\mathbb{Z}^2_N,p_d}$. We derived bounds on the diameter of this graph and described its degree distribution.
We studied the activation processes on $G_{\mathbb{Z}^2_N,p_d}$ and approximated the evolution of the density in the real model with critical values in mean-field. Specifically, we derived conditions for phase
transitions as a function of initialization probability $p$ and long edge parameter $\lambda$. The dependence of $p_c$ on $\lambda$  in the mean field model and numerical values for $p^{(crit)}_{in}(\lambda )$ in the real model (if exist) are shown on Figures~\ref{pcLambda},  \ref{pcLambdaReal}, respectively. It seems that $p_c$ approximates $p^{(crit)}_{in}(\lambda)$ well.
%One can also see from Figure~\ref{pcLambda} that $p_c$ drops significantly for $\lambda \in (0.1, 10)$.

%\textcolor{red}
{The model introduced in this paper is motivated by the structure and operation of the neuropil, the densely connected neural tissue of the cortex \cite{freeman91, kozma05}. The human brain has about \(10^{11}\) neurons. Typically, a neuron has several thousands of connections to other neurons through synapses, thus the human brain has \(\sim 10^{15}\) synaptic
connections. Most of the connections are short and limited to the neuron's
direct neighborhood (in some metric), forming the so-called the {\em dendritic arbor}. In addition, the neurons have a few long connections (axons), which extend further away from their cell body. In general, there are several thousands short connections in the dendritic arbor for a few distant connections represented by long axons. We use $G_{\mathbb{Z}^2_N,p_d}$ to model the combined effect of mostly short connections  and a few long connections.
It is much more likely to have in brains shorter connections than longer ones,
which is a fact captured in the definition of $p_d$, as $p_d$ is decreasing in the graph distance $d$.}

%\textcolor{red}
{
There are two types of neurons in the brain, namely excitatory and inhibitory ones. The type of a neuron describes the function of the neuron in the brain. Excitatory (inhibitory) neurons excite (inhibit) the neurons to which they are connected.
It is known that there are much more excitatory neurons than inhibitory neurons in the cortex; the ratio of inhibitory to excitatory neurons is typically \(1/4\) \cite{freeman91}.
Based on neuroscience studies, it is expected that pure excitatory populations can maintain non-zero background activation level, while interacting excitatory and inhibitory populations are able to produce limit cycle oscillations \cite{kozma15}.

%\textcolor{red}
{
This paper focuses on conditions required to sustain non-zero activity level in pure excitatory networks, but the model can be generalized to include two types of vertices \cite{krs}. Here we briefly outline the proposed approach.
The type of a vertex is either excitatory (\(E\)) or inhibitory (\(I\)). Let $A_{E}(t)$ and $A_{I}(t)$ be the sets of active vertices of type \(E\) and \(I\) at time \(t\), respectively. The total number of active vertices is given by $A(t) = A_{E}(t) \cup A_{I}(t)$. Using the potential function defined in Section~\ref{section3}, we can rewrite $A_i(t) = \{v \in
V(G_{\mathbb{Z}^2_N,p_d})  \bigm| \chi_v (t)=1 \  \&~v \text{ is of }\text{ type}~i \}$, where $i \in \{E, I\}$. Each vertex may change its activity based on the states of its neighbors. We define the modified $k$-threshold rule for two types of vertices as follows. For a vertex $v$ of type \(E\), the evolution rule is
}
%\textcolor{red}
{
\begin{equation}
\chi_v (t+1) = \mathbbm{1}  \left(
\sum_{u \in N^{E}(v)}{\chi_u(t)} - \sum_{u \in N^{I}(v)}{\chi_u(t)} \geq k
\right),
\label{dynamicsE}
\end{equation}
where $N^{E}(v)$  and $N^{I}(v)$ denote the subsets of vertices in the closed neighborhood of the vertex $v$, of type \(E\) and \(I\), respectively.
For a vertex $v$ of type \(I\), the following rule holds:
\begin{equation}
\chi_v (t+1) = \mathbbm{1}  \left(
\sum_{u \in N^{E}(v)}{\chi_u(t)} + \sum_{u \in N^{I}(v)}{\chi_u(t)} \geq k
\right) = \mathbbm{1}  \left( \sum_{u \in N(v)}{\chi_u(t)} \geq k \right),
\label{dynamicsI}
\end{equation}
where $N(v) = N^{E} (v) \cup N^{I} (v)$ is the closed neighborhood of vertex
$v$. Notice, that vertices of type \(E\) and \(I\) influence each other differently.
}

 Open problems concerning the properties of $G_{\mathbb{Z}^2_N,p_d}$ and the activation process on the graph include: What is the number of small cycles? What is the clustering coefficient? Does a unique limit density
 defined in Conjectures \ref{c1}, \ref{c2} and \ref{c3} exist? 
Additional open questions include the generalization of these results for other lattice types and higher dimensions. 

\begin{acknowledgement*} We are grateful to the unknown referee whose important suggestions led to additional results and improved the presentation of this paper. Figures \ref{conjecturelambda} and 
\ref{pcLambdaReal} are based on simulations that were done by Gabriel P. Andrade.
\end{acknowledgement*}

\section{Appendix. }
\subsection{The second derivative of $\log(P_k(x))$}
\begin{equation}
\left( \log(P_2(x)) \right)''  =  \frac{-5( \lambda^4(1-x)^4 +16 \lambda^3 (1-x)^3 + 96 \lambda^2 (1-x)^2 + 240 \lambda (1-x) + 240)}
{(1-x)^2 ( \lambda^2(1-x)^2 +10\lambda (1-x) + 20)^2},
\end{equation}
\begin{multline}
\left( \log(P_3(x)) \right)''  = 
-5 \left[ \lambda^6 (1-x)^6 + 24 \lambda^5 (1-x)^5 \right.
\\
\left. +228 \lambda^4 (1-x)^4+ 1056 \lambda^3 (1-x)^3 +2520 \lambda^2 (1-x)^2 +2880 \lambda (1-x) + 1440) \right] 
\\
\left( (1-x) ( \lambda^3(1-x)^3 +15 \lambda^2 (1-x)^2 +60 \lambda (1-x) + 60) \right)^{-2},
% \frac{-5(1-x)^2 \left[1440+(1-x)^6\lambda^6 + 24(1-x)^5\lambda^5 \right.}
%{}
%\\
%\\
%\frac{ \left. +228(1-x)^4\lambda^4+ 1056(1-x)^3\lambda^3+2520(1-x)^2\lambda^2+2880(1-x)\lambda) \right]}{(1-x)^4(-60+(1-x)^3\lambda^3+15(1-x)^2\lambda^2+60(1-x)\lambda)^2},
\end{multline}
\begin{multline}
\left( \log(P_4(x)) \right)''  =  -5 
\left[ \lambda^8 (1-x)^8 + 32 \lambda^7 (1-x)^7 + 416 \lambda^6 (1-x)^6 + 2784 \lambda^5 (1-x)^5 \right.
\\
\left.  + 10320 \lambda^4 (1-x)^4 + 21120 \lambda^3 (1-x)^3 + 23040 \lambda^2 (1-x)^2 + 11520 \lambda(1-x) +2880  \right]
\\
\left( (1-x) ( \lambda^4 (1-x)^4 +20 \lambda^3 (1-x)^3 +120 \lambda^2 (1-x)^2 +240 \lambda (1-x) + 120) \right)^{-2},
%%\frac{-14400-5(1-x)^8\lambda^8 - 160(1-x)^7\lambda^7 -2080(1-x)^6\lambda^6 }
%%{} 
%%\\
%%\\
%%\frac{ - 13920(1-x)^5 \lambda^5 -51600 (1-x)^4 \lambda^4+105600(1-x)^3\lambda^3 -115200 (1-x)^2 \lambda^2 - 57600(1-x)\lambda }
%%{(1-x)^2 (120+(1-x)^4 \lambda^4 +20(1-x)^3\lambda^3+120(1-x)^2 \lambda^2+240 (1-x)\lambda)^2},
\end{multline}
\begin{multline}
\left( \log(P_5(x)) \right)''  = 
-5 \lambda^2 \left[ \lambda^8(1-x)^8 + 40\lambda^7 (1-x)^7+660\lambda^6 (1-x)^6 + 5760 \lambda^5(1-x)^5  \right. 
\\
\left. + 28800\lambda^4(1-x)^4 + 83520\lambda^3 (1-x)^3 +136800 \lambda^2 (1-x)^2 + 115200 \lambda(1-x) + 43200 \right]
\\
(\lambda^5 (1-x)^5 + 25 \lambda^4(1-x)^4 +200 \lambda^3 (1-x)^3 + 600\lambda^2(1-x)^2 +600 \lambda (1-x) + 120)^{-2}.
\end{multline}

%%%%% -4 \lambda^2 \left[ 7200+(1-x)^6  \lambda^6 +30(1-x)^5 \lambda^5+360(1-x)^4 \lambda^4  \right. 
%%%%%\\
%%%%%\left. + 2160(1-x)^3 \lambda^3 +6840(1-x)^2 \lambda^2+10800(1-x) \lambda \right]
%%%%%\\
%%%%%\left( 120+(1-x)^4 \lambda^4 +20(1-x)^3 \lambda^3+120(1-x)^2 \lambda^2+240(1-x) \lambda \right)^{-2}.
%% \frac{-4 \lambda^2 \left[ 7200+(1-x)^6  \lambda^6 +30(1-x)^5 \lambda^5+360(1-x)^4 \lambda^4  \right. } 
%%{}
%%\\
%%\\
%%\frac{ \left. + 2160(1-x)^3 \lambda^3 +6840(1-x)^2 \lambda^2+10800(1-x) \lambda \right]} {(120+(1-x)^4 \lambda^4 +20(1-x)^3 \lambda^3+120(1-x)^2 \lambda^2+240(1-x) \lambda)^2}.
%%\end{multline}
Clearly, $\left( \log(P_i(x)) \right)''<0$, for $i=2,3,4,5$, $\lambda > 0$, $x\in (0,1)$.

\subsection{Description of simulations}
Figure~1 and 3 in the main text are based on the simulations of the real process on $G_{\mathbb{Z}^2_N,p_d}$ with $N=100$, i.e., $|V(G_{\mathbb{Z}^2_N,p_d})| = 10000$, that were done as follows. For every value of $\lambda$, 15 graphs were generated, each with different random seed. The process was ran on each graph for all initialization probabilities between 0 and 1 with step 0.01. "Limiting" densities shown on Figure~1 were obtained under the condition that either the density converges after the first 1000 iterations to 0 or 1, or undergoes repetitions after the first 1000 iterations.

\end{document}